\documentclass[11pt]{amsart}
\usepackage{amsfonts}
\usepackage{indentfirst,amsthm,bm,amsmath}
\usepackage{hyperref}
\makeatletter
\date{}
\newtheorem{theorem}{Theorem}[section]
\newtheorem{corollary}{Corollary}[section]
\newtheorem{lemma}{Lemma}[section]

\newtheorem{definition}{Definition}[section]

\newtheorem{example}{Example}[section]
\newtheorem{question}{Question}[section]

\begin{document}
\title[]{Nevanlinna theory for Jackson difference operators and entire solutions of $q$-difference equations}

\author[Tingbin Cao]{Tingbin Cao}
\address[Tingbin Cao]{Department of Mathematics, Nanchang University, Nanchang, Jiangxi 330031, P. R. China}
\email{tbcao@ncu.edu.cn}
\thanks{The first author is supported by the National Natural Science Foundation of China
(No. 11871260, No. 11461042) and the outstanding young talent assistance program of Jiangxi Province (No. 20171BCB23002) in China.}

\author[Huixin Dai]{Huixin Dai}
\address[Huixin Dai]{Department of Mathematics, Nanchang University, Nanchang, Jiangxi 330031, P. R. China}
\curraddr{Department of Mathematics, Beijing University of Posts and Telecommunications, Beijing 100876, P. R. China}
\email{dhxdai@163.com}

\author[Jun Wang]{Jun Wang}
\address[Jun Wang]{School of Mathematic Sciences, Fudan University, Shanghai 200433, P. R. China}
\email{majwang@fudan.edu.cn}
\thanks{The third author is supported by the National Natural Science Foundation of China (No. 11771090).}

\subjclass[2010]{Primary 30D35; Secondary 39A70, 33D99, 39A13}
\keywords{Jackson difference operators; Nevanlinna theory; entire functions; $q$-difference equations; $q$-special functions}

\begin{abstract} This paper establishes a version of Nevanlinna theory based on Jackson difference operator $D_{q}f(z)=\frac{f(qz)-f(z)}{qz-z}$ for meromorphic functions of zero order in the complex plane $\mathbb{C}$. We give the logarithmic difference lemma, the second fundamental theorem, the defect relation, Picard theorem and five-value theorem in sense of Jackson $q$-difference operator. By using this theory, we investigate the growth of entire solutions of linear Jackson $q$-difference equations $D^{k}_{q}f(z)+A(z)f(z)=0$ with meromorphic coefficient $A,$ where $D^k_q$ is Jackson $k$-th order difference operator, and estimate the logarithmic order of some $q$-special functions.
\end{abstract}
\maketitle
\tableofcontents

\section{Introduction}
The study on $q$-functions and  $q$-difference equations appeared already at the beginning of the last century, see works by Jackson \cite{Jackson, Jackson-2}, Carmichael \cite{carmichael}, Mason \cite{Mason}, Trjitzinskey \cite{trjitzinskey} and other known authors such as Euler, Poincare, Picard, Ramanunjan. Birkhoff and Guenther \cite{birkhoff-guenther} once announced a program which they did not develop further, and $q$-difference equations remained less advanced than differential equations and difference equations. Since years eighties \cite{hahn-2}, an intensive and somewhat surprising interest in the subject reappeared in mathematics and its applications. Mathematicians have reconsidered $q$-difference equations for their links with other branches of mathematics such as quantum algebras and $q$-combinatorics, and Birkhoff and Gunther's program has been continued. For examples, B\'{e}zivin and Ramis' results on divergent seires have appplications to rationality criteria for solutions of systems of $q$-difference equations \cite{bezivin-boutabaa} and for systems of $q$-difference and differential equations \cite{ramis};  L. D. Vizio \cite{vizio} studied the $q$-analogue of Grothendieck-Katz's conjecutre on $p$-curvatures on the arithmetic theory of $q$-difference equations; Z. G. Liu \cite{liuzhiguo} investigated the $q$-partial differential equations and $q$-series; J. Cao \cite{caojian} considered homogeneous $q$-partial difference equations.\par

Let $s\in \mathbb{Z}$ and $0<|q|<1.$ The subjacent theory was founded on the corresponding divided difference derivative \cite{magnus, magnusap, marco-parcet} as follows
\begin{eqnarray}\mathcal{D}f(x(s))=\frac{f(x(s+\frac{1}{2}))-f(x(s-\frac{1}{2}))}{x(s+\frac{1}{2})-x(s-\frac{1}{2})}.\end{eqnarray}
The basic property of this derivative is that it sends a polynomial of degree $n$ to a polynomial of degree $n-1.$ \par

(I). If the lattice $x(s)$ is a constant, then the corresponding divided derivative gives just the classical derivative $D(f(x))=\frac{d}{dx}f(x).$\vskip 2mm

(II). If $x(s)$ is the special lattice of $x(s)=s,$ then the divided derivative gives the classical difference
\begin{eqnarray*}\Delta_{\frac{1}{2}}f(x)=\Delta f(t)=f(t+1)-f(t),\,\,\, t=x-\frac{1}{2}.\end{eqnarray*}\par

(III). If $x(s)=q^s,$ the divided derivative  yields the so-called Jackson difference operator (or called Jackson $q$-derivative)\cite{Jackson1908, Jackson, Jackson-2}
\begin{eqnarray*}D_{q^{\frac{1}{2}}}=D_{q}f(t)=\frac{f(qt)-f(t)}{qt-t}, \,\,\, t=q^{-\frac{1}{2}}x.\end{eqnarray*}

(IV). If $x(s)=\frac{q^{s}+q^{-s}}{2},$ then the derivative is the so-called Askey-Wilson divided difference operator \cite{R.Askey} that can be written as
\begin{eqnarray*}D_{AW}f(x(z))=\frac{f(x(q^{\frac{1}{2}}z))-f(x(q^{-\frac{1}{2}}z))}{x(q^{\frac{1}{2}}z)-x(q^{-\frac{1}{2}}z)},\quad z=q^{s}.
\end{eqnarray*}
What's more, Wilson also proposed the concept of the $Wilson$ difference operator
to study Wilson polynomials $W_n(x;a,b,c,d)$, see \cite{R.Askey}.
\vskip 2mm
\par
It is well-known that the Nevanlinna theory \cite{hayman} based on the classical derivative operator was established by R. Nevanlinna in the 1920s. It has been played the key role in studying oscillation of complex differential equations \cite{laine}. Recently, the Nevanlinna theory on some divided difference derivatives was investigated. For classical difference operator $\Delta f(x)=f(x+c)-f(x)$,  its Nevanlinna theory was firstly discussed by Halburd-Korhonen \cite{22, 23} and Chiang-Feng \cite{Feng, chiang-feng-2009} independently. Chiang and Feng \cite{Wilson2} considered the Nevanlinna theory for the Askey-Wilson difference operator, and that for Wilson difference operator was studied by Cheng and Chiang \cite{Wilson}.
Meanwhile, Nevanlinna theory for these difference operators have been positively applying to study complex difference equations. Now it remains to be seen its own version for the Jackson difference operator appeared in (III)  so that being applied to study complex Jackson $q$-difference equations.
\vskip 2mm
\par

Recall that the Jackson difference operator
\begin{eqnarray}D_{q}f(z)=\frac{f(qz)-f(z)}{qz-z},\,\,\, z\in\mathbb{C}, \,\, 0<|q|<1,\end{eqnarray}  was initially investigated by Jackson \cite{Jackson1908, Jackson, Jackson-2} in 1908. Very recently, the authors learn from Professor Zhiguo Liu \footnote{On the Chinese-Finnish Workshop in Complex Analysis 2019 at Suzhou University of Science and Technology} that in fact, this concept can be traced back to L. Schendel \cite{schendel} in 1877. Clearly, if $f$ is differentiable, $\lim_{q\to 1}D_{q}f(z)=\frac{d}{dz}f(z)$. Observing that $D_{q}z^{k}=\frac{q^{k}-1}{q-1}z^{k-1},$  we have $$D_{q}\Big(\sum_{j=0}^na_jz^j\Big)=\sum_{j=0}^{n-1}a_{j+1}\frac{q^{j+1}-1}{q-1}z^{j}.$$ Furthermore, the Jackson difference operator has the derivative rules of product, ratio, chain rule, inverse function and Leibniz formula similar to that of the classical derivative $\frac{d}{dz}$, which we will show in Lemma \ref{L4.1} later.
\vskip 2mm
\par To discuss the solutions $f(z)=\sum_{n=0}^{\infty}c_{n}z^{n}$ of $q$-difference equations, we ecall the following notations (refer to \cite{bangerezako}), for $a\in\mathbb{C},n\in\mathbb{N}$,
$$(a; q)_{0}=1,\quad (a; q)_{n}=(1-a)(1-aq)(1-aq^{2})\cdots(1-aq^{n-1}).$$
If $0<|q|<1,$ then $(a; q)_{\infty}:=\prod_{n=0}^{+\infty}(1-aq^{n}).$ Define
$$(a_{1}, \ldots, a_{p}; q)_{n}:=(a_{1}; q)_{n}\cdots(a_{p}; q)_{n},\quad \left[\begin{array}{c}
                                                                                             n \\
                                                                                            j
                                                                                          \end{array}
\right]_{q}=\frac{(q; q)_{n}}{(q; q)_{j}(q; q)_{n-j}}.$$
It is of particular interest considering the case that $c_{n+1}/c_{n}$ is a rational function in $q^{n}$. For example
\begin{eqnarray*}\frac{c_{n+1}}{c_{n}}=\frac{\prod_{j=1}^{r}(\alpha_{j}-q^{-n})}{\prod_{j=1}^{s}(\beta_{j}-q^{-n})(q-q^{-n})},
\end{eqnarray*} such series seems to have the form
\begin{eqnarray}
&&_r\phi_{s}\left(\left.\begin{array}{cccc}
                 \alpha_{1} & \alpha_{2} & \ldots, & \alpha_{r} \\\nonumber
                 \beta_{1} & \beta_{2} & \ldots, & \beta_{s}
               \end{array}\right| q; z
\right)\\
&=&\sum_{j=0}^{\infty}\frac{(\alpha_{1}, q)_{j}(\alpha_{2}; q)_{j}\cdots(\alpha_{r}; q)_{j}}{(\beta_{1}, q)_{j}(\beta_{2}; q)_{j}\cdots(\beta_{s}; q)_{j}}
\left[(-1)^{j}q^{\frac{j(j-1)}{2}}\right]^{1+s-r}\frac{z^{j}}{(q; q)_{j}}.\end{eqnarray} These series are referred to as the $q$-(basic) hypergeometric series \cite{gasper-rahman}. It is known \cite[Pages 17-18]{bangerezako} that every nonzero solution of the first order Jackson $q$-difference equation $D_{q}f(z)=f(z)\,(0<|q|<1)$ reads
$$f(z)=_1\phi_{0}(0;-;q,(1-q)z)=\sum_{n=0}^{\infty}\left(\prod_{j=1}^{n}\frac{1-q}{1-q^{j}}\right)z^{n}=\sum_{n=0}^{\infty}\frac{(1-q)^{n}z^{n}}{(q; q)_{n}}$$
which is also named by $\exp_{q}(z)$ (sometimes, we also use $e_{q}^{z}$), the $q$-version exponential function (see Example \ref{Ex-1}). The nonzero solution of Jackson $q$-difference equation of the first order form $$D_{q}f(z)=af(z), \,\,\, a\in\mathbb{C}\setminus\{0\}, \,\,\, 0<|q|<1$$ is $c_{0}\exp_{q}(az)$ where $c_{0}$ is the constant term of the expand series of $f$ at origin. For $k\in\mathbb{N}\cup\{0\}$, we denote by Jackson $k$th-order difference operator \begin{eqnarray}D_q^0f(z):=f(z),\quad D_{q}^{k}f(z):= D_{q}(D_{q}^{k-1}f(z)).
\end{eqnarray} For more background of Jackson difference operators and $q$-difference equations \begin{eqnarray*}F(z, f(z), D_{q}f(z), \ldots, D^{k}_{q}f(z))=0,\end{eqnarray*} we refer to see the book \cite{bangerezako}. \vskip 2mm
\par These rich background and recent works on Nevanlinna theory \cite{Wilson, Wilson2, chiang-feng-2016} motivate us to study the Nevanlinna theory and $q$-difference equations for Jackson difference operators. To do that, we will apply the corresponding results for the classical differential operator \cite{hayman} and the $q$-difference operator $\nabla_{q}f(z)=f(qz)-f(z)$ \cite{Barnett}.
This paper is organised as follows.\vskip 2mm
\par Section 2 first gives some basic notions and results in classical Nevanlinna theory, then shows the logarithmic derivative lemma, the second fundamental theorem, defect relation, Picard theorem  and five-value theorem for Jackson difference operator
(Theorems 2.1-2.6).
In Section 3, we consider the Jackson Kernel $Ker(D_{q}),$ and show an interesting phenomenon (Theorem \ref{T6}) that the Jackson $q$-Casorati determinant $C_{J}(f_{1}, f_{2})$ does not belong to $Ker(D_q)$, where $f_1,f_2$ are two linearly independent analytic solutions at the origin of the linear Jackson $q$-difference equation $D_{q}^{2}f(z)+A(z)f(z)=0$. This is very different from the case of derivative operator in the differential equations \cite{laine}. Section 4 mainly investigates the growth of entire solutions of linear Jackson $q$-difference equation $D_{q}^{k}f(z)+A(z)f(z)=0$. Several examples are given to explain that Theorem \ref{T5} can help us to know the exact logarithmic order of some known $q$-special functions.\par

\section{Nevanlinna theory for Jackson difference operator}

Before establishing Nevanlinna theory for Jackson difference operators, for convenience of readers, we briefly introduce the basic notation and results of classical Nevanlinna theory for derivative operator $\frac{d}{dz}$.
\subsection{Preliminaries of classical Nevanlinna theory}
 Let $f(z)$ be a nonconstant meromorphic function on $\mathbb{C}$. For $r>0$, we denote $\log^{+}r=\max\{\log r,0\}$. The Nevanlinna characteristic of $f$ is defined to be the real-valued function
\begin{equation}\label{2.1}
T(r,f):=m(r,f)+N(r,f),
\end{equation}
where $m(r,f)$ and $N(r,f)$ are called the proximity function and counting function respectively, and $m(r,f)=\int_0^{2\pi}\log^+|f(re^{i\theta})|d\theta,$
\begin{equation}\label{2.2}
N(r,f)=\int^{r}_{0}\frac{n(t,f)-n(0, f)}{t}dt+n(0, f)\log r.
\end{equation}
Here, $n(t, f)$ denotes the number of poles of $f$ in $\{|z|<t\}$ counting multiplicities.
 The characteristic function $T(r,f)$ is an increasing convex function of $\log r$, which plays the role of $\log M(r,f)$ for an entire function. The order of $f$ is defined by
\begin{equation}\sigma(f)=\limsup_{r\to\infty}\frac{\log T(r,f)}{\log r}.\end{equation}
\par The \emph{first fundamental theorem}  states that for any complex number $a\in\mathbb{C}\cup\{\infty\}$
\begin{equation}\label{2.3}
T(r,\frac{1}{f-a})=T(r,f)+O(1)
\end{equation}
as $r$$\rightarrow+\infty,$ which comes from the Jensen formula
\begin{equation}\label{2.3}
N(r, \frac{1}{f})-N(r,f)=\frac{1}{2\pi}\int_{0}^{2\pi}\log |f(re^{i\theta})|d\theta.
\end{equation}Denote by $S(r, f)$ the quantity of $S(r,f)=o(T(r,f))$ possibly outside a exceptional set in $r$ of finite linear measure   and by $\overline{N}(r, f)$ the counting function defined by $\overline{n}(t, f)$ the number of poles of $f$ ignoring multiplicities. In 1925, R. Nevanlinna established \emph{the second fundamental theorem} that for any $p$ distinct values $a_1,\cdots,a_p\in \mathbb{C}\cup\{\infty\}$,
\begin{eqnarray*}(p-2)T(r, f)&<&\sum_{j=1}^{p}N(r, \frac{1}{f-a_{j}})-N_{1}(r)+S(r, f)\\
&\leq&\sum_{j=1}^{p}\overline{N}(r, \frac{1}{f-a_{j}})+S(r, f),\end{eqnarray*} where $N_{1}(r)=2N(r, f)-N(r, f')+N(r, 1/f')$. It was proved by the logarithmic derivative lemma that $m(r, f^{'}/f)=S(r,f)$, which is also useful in the study on complex difference equations \cite{laine}. We refer the readers to see the well-known book due to Hayman \cite{hayman} for the details of classical Nevanlinna theory.

\subsection{Jackson difference analogue of logarithmic derivative lemma} Without loss of generality, set $0<|q|<1.$ We now consider Jackson difference operator
\begin{equation}
D_{q}f(z)=\frac{f(qz)-f(z)}{qz-z}.
\end{equation}
Based on the $q$-analogue of logarithmic derivative lemma \cite{Barnett}, we obtain the the logarithmic derivative lemma for Jackson difference operators as follows.\par

\begin{theorem} \label{T1}
Let $f$ be a nonconstant  meromorphic function with zero order. Then
\begin{equation}
m\Big(r,\frac{D_{q}^{k}f(z)}{f(z)}\Big)=o(T(r,f))
\end{equation}
on a set of logarithmic density 1.
\end{theorem}
\begin{proof} By \cite[Theorem 1.1]{Barnett} and \cite[Theorem 1.1]{zhang-korhonen}, if $f$ is a nonconstant meromorphic function with zero order,and $q\in\mathbb{C}\setminus\{0\}$, then
\begin{equation}\label{E1.3}
m\Big(r,\frac{f(qz)}{f(z)}\Big)=o(T(r,f)),\quad T(r, f(qz))=T(r, f)+o(T(r,f))
\end{equation}
hold for all $r$ on a set of logarithmic density $1.$ Thus for every $0\leq j\leq k$, we have\begin{equation}
m\Big(r,\frac{f(q^{k-j}z)}{f(z)}\Big)\leq \sum_{n=0}^{k-j-1}m\Big(r,\frac{f(q^{n+1}z)}{f(q^{n}z)}\Big)=o(T(r,f)),
\end{equation} for all $r$ on a set of logarithmic density $1.$ Hence,  \begin{eqnarray*}
m\Big(r,\frac{D_{q}(f)}{f(z)}\Big)\leq m\Big(r,\frac{f(qz)}{f(z)}\Big)+m\Big(r,\frac{1}{(q-1)z}\Big)+O(1)=o(T(r,f))
\end{eqnarray*}holds for all $r$ on a set of logarithmic density $1.$ For general positive integer $k>2$, it follows from the equality \cite[page 13]{bangerezako} \begin{eqnarray*}D^{k}_{q}f(z)=(q-1)^{-k}z^{-k}q^{-k(k-1)/2}\sum_{j=0}^{k}(-1)^{j}\left[\begin{array}{c}
                                                                                            k \\
                                                                                            j
                                                                                          \end{array}
\right]_{q}q^{j(j-1)/2}f(q^{k-j}z),\end{eqnarray*} again by (13), we have
\begin{eqnarray*}m(r, \frac{D^{k}_{q}f(z)}{f(z)})&\leq& \sum_{j=0}^{k}m\Big(r, \frac{f({q}^{k-j}z)}{f(z)}\Big)+O(1)=o(T(r, f))
\end{eqnarray*}for all $r$ on a set of logarithmic density $1.$ This completes the proof.\end{proof}

\subsection{Second fundamental theorem for Jackson difference operator} For $a\in\mathbb{C}$,
$\overline{n}(r, \frac{1}{f-a})$ can be written as a sum of integers $``h-k"$ summing over all the zeros of $f(z)-a$ in $\{z: |z|<r\}$  with multiplicity $``h",$ and where $``k(=h-1)"$ is the multiplicity of $f'(z)=0$ where $f(z)=a$. Similarly, $\overline{n}(r, f)=\overline{n}(r, \frac{1}{f}=0)$ can be written as a sum of integers $``h-k"$ summing over all the poles of $f$ in $\{z: |z|<r\}$ with multiplicity $``h",$ and where $``k(=h-1)"$ is the multiplicity of $d\big(\frac{1}{f(z)}\big)/dz=-f'(z)/f^{2}(z)=0$ where $f(z)=\infty.$\vskip 2mm
\par
We define a Jackson analogue of the $\bar{n}(r, \frac{1}{f-a})$ and $\bar{n}(r, f),$ similarly as in \cite{Wilson2, Wilson}. Denote
$$\tilde{n}_{J}(r,\frac{1}{f-a})=\tilde{n}_{J}(r, f=a)$$ to be the sum of the form $``h-k"$ summing over all the points $z$ in ${|z|<r}$ at which $f(z)=a$ with multiplicity ``$h$", while the $``k"$ is defined by $k:=\min\{h,k'\}$, $k'$ is the multiplicity of $D_{q}f(z)=0$ at $z.$ Recall that the Jackson difference operator sends a polynomial of degree $n$ to a polynomial of degree $n-1,$ then $\tilde{n}_{J}(r,p=a)=1$ holds for any nonconstant polynomial function $p(z).$ Thus it is given in a natural way as in classical Nevanlinna theory. And $$\tilde{n}_{J}(r, f)=\tilde{n}_{J}(r, f=\infty)=\tilde{n}_{J}(r, \frac{1}{f}=0)$$ can be written as a sum of integers $``h-k"$ summing over all the points $z$ in $\{z: |z|<r\}$ at which $f(z)=\infty$ with multiplicity $``h",$ while $k:=\min\{h,k'\}$ and $k'$ is the multiplicity of $D_{q}\big(\frac{1}{f(z)}\big)=-\frac{D^{q}f(z)}{f(z)f(qz)}=0$ where $f(z)=\infty.$ Then for any $a\in\mathbb{C}\cup\{\infty\}$, we define the Jackson-type counting functions as $$\tilde{N}_{J}(r,f=a)=\int_{0}^{r}\frac{\tilde{n}_{J}(t,f=a)-\tilde{n}_{J}(0,f=a)}{t}dt+\tilde{n}_{J}(0,f=a)\log r.$$
Since the truncated counting function $\tilde{N}_{q}(r, \frac{1}{f-a})$ defined in \cite{Barnett} is possible negative for all $r$, the Jackson-type counting function $\tilde{N}_{J}(r,f=a)$ is better than $\tilde{N}_{q}(r, \frac{1}{f-a}).$ \vskip 2mm
\par
Next, we will deduce the second fundamental theorem in terms of Jackson-type counting function, which is based on the second fundamental theorem due to Barnett-Halburd-Korhonen-Morgan \cite{Barnett}. Of course, this can also proved directly in terms of the logarithmic difference lemma for Jackson difference operator (Theorem \ref{T1}), similarly as in \cite{hayman, Barnett, Wilson, Wilson2}.\par

\begin{theorem}\label{T2}
Let $f$ be a nonconstant meromorphic function of zero order, let $0<|q|<1,$ and let $a_{1}, \ldots, a_{p}$ $(p\geq 2)$ be distinct points in $\mathbb{C}\cup\{\infty\}.$ Then \begin{eqnarray}\label{E3.5}
(p-2)T(r,f)&\leq& \sum_{j=1}^{p}N(r, f=a_{j})-N_{J}(r)-\log r+o(T(r,f))\\\nonumber
&\leq& \sum_{j=1}^{p}\tilde{N}_{J}(r, f=a_{j})+o(T(r,f))\end{eqnarray}
holds for all $r$ on a set of logarithmic density one,  where $$N_{J}(r)=2N(r, f)-N(r, D_{q}(f))+N(r, \frac{1}{D_{q}f}).$$\end{theorem}

\begin{proof}
Since $f$ is not constant, we have $D_{q}f\not\equiv 0.$ Otherwise, $f(z)\equiv f(q^{k}z)$ for every $k\in\mathbb{N}$. Since $q^{k}\to 0$ as $k\to\infty,$ we get that $f(z)\equiv f(0)$, which is impossible. Hence we have $\nabla_{q}f:=f(qz)-f(z)\not\equiv 0.$ It follows from \cite[Theorem 3.1]{Barnett} that
\begin{eqnarray} (p-2)T(r,f)\leq \sum_{j=1}^{p}N(r, f=a_{j})-N_{\nabla_q}(r)+o(T(r,f))\end{eqnarray}
holds for all $r$ on a set of logarithmic density $1$, where $$N_{\nabla_q}(r)=2N(r, f)-N(r, \nabla_{q} f)+N(r,\frac{1}{\nabla_{q} f}).$$ Since $D_{q}f(z)=\frac{\nabla_{q}f(z)}{(q-1)z}$, it follows from the Jensen formula that  \begin{eqnarray*}
N(r,\frac{1}{D_{q}f})-N(r,D_{q}(f))&=&\int_{0}^{2\pi}\log |D_{q}f(re^{i\theta})|\frac{d\theta}{2\pi}+O(1)\\
&=&\int_{0}^{2\pi}\log |\nabla_{q}f(re^{i\theta})|\frac{d\theta}{2\pi}-\log r+O(1)\\
&=&N(r, \frac{1}{\nabla_{q} f})-N(r,\nabla_{q} f)-\log r+O(1).\end{eqnarray*} Hence, we have
$$N_{\nabla_q}(r)=2N(r, f)-N(r,D_{q}(f))+N(r,1/D_{q}f)+\log r+O(1).$$
Taking it into (15) yields the first inequality in (14), that is,
\begin{equation}
(q-2)T(r,f)\leq \sum_{j=1}^p N(r,f=a_j)-N_J(r)-\log r+o(T(r,f))
\end{equation}
holds in a set of $r$ with logarithmic density one.\vskip 2mm
\par From the definition of $\tilde{n}_{J}(r, f=a)$, when $a\in\mathbb{C}$, the difference between $n(r, f=a)$ and $\tilde{n}_{J}(r, f=a)$ happens at zeros of $D_{q}f(z)$ at which $f-a$ has a zero in the disk $|z|<r$. If $a=\infty$, $n(r, f=\infty)-\tilde{n}_{J}(r, f=\infty)$ enumerates at most the number of zeros of $D_{q}(\frac{1}{f(z)})$ at which $f(z)$ has a pole in the disk $|z|<r$, with due count of multiplicities. Since
\begin{equation}\label{20}
D_{q}\Big(\frac{1}{f(z)}\Big)=\frac{-D_{q}f(z)}{f(qz)f(z)},\end{equation}
the zeros of $D_{q}(\frac{1}{f(z)})$ originate from the poles of $f(qz)$,$f(z)$ or from the zeros of $D_{q}f(z)$. We note that the poles of $D_{q}f(z)$ must be among the poles of $f(z),f(qz)$ and the origin with simple multiplicity. Thus, the multiplicity of zeros of $D_{q}\frac{1}{f(z)}$ is no more than the sum of multiplicities of the poles of $f(z),f(qz)$ subtracting the multiplicity of poles of $D_{q}f(z)$. We will add $1$ to the upper bound when $z=0$ is one pole of $D_{q}f(z).$ Therefore, it follows from the above discussions that for distinct values $a_1,a_2,\cdots,a_p\in \mathbb{C}\cup\{\infty\}$,
\begin{equation}
\begin{split}\label{E1.1}
\sum_{j=1}^{p}&\left(N(r, f=a_{j})-\tilde{N}_{J}(r, f=a_{j})\right)\\
&\leq N(r, f(z)=\infty)+N(r,f(qz)=\infty)+N(r, D_{q}f(z)=0)\\
&\quad -N(r, D_{q}f(x)=\infty)+\log r.
\end{split}\end{equation}
\cite[Theorem 1.3]{zhang-korhonen} says that for a meromorphic function $f$ with zero order,
\begin{equation}\label{E1.2}N(r, f(qz)=\infty)=(1+o(1))N(r, f(z)=\infty)
\end{equation}on a set of lower logarithmic density one. Then combining \eqref{E1.1} with \eqref{E1.2} follows
\begin{equation*}\begin{split}
\sum_{j=1}^{p}N(r, f=a_{j})&\leq \sum_{j=1}^{p}\tilde{N}_{J}(r, f=a_{j})+(2+o(1))N(r, f(z)=\infty)\\
&\quad +N(r, D_{q}f(z)=0)-N(r, D_{q}f(x)=\infty)+\log r\\
&=\sum_{j=1}^{p}\tilde{N}_{J}(r, f=a_{j})+N_{J}(r)+\log r+o(T(r, f)).
\end{split}\end{equation*}
Submitting this into (16), we get the conclusion of this theorem.
\end{proof}

\subsection{Defect relation for Jackson difference operator} For a given meromorphic function $f,$ the Nevanlinna defect $\delta(a, f),$ multiplicity index $\vartheta(a,f)$ and ramification index $\Theta(a,f)$ of $f$ at $a\in\mathbb{C}\cup\{\infty\}$ are defined respectively as $$\delta(a, f):=1-\limsup_{r\rightarrow\infty}\frac{N(r, \frac{1}{f-a})}{T(r, f)},\quad \vartheta(a, f):=\liminf_{r\rightarrow\infty}\frac{N(r, \frac{1}{f-a})-\overline{N}(r, \frac{1}{f-a})}{T(r, f)},$$
and
$$\Theta(a, f):=1-\limsup_{r\rightarrow\infty}\frac{\overline{N}(r, \frac{1}{f-a})}{T(r, f)}.$$ It follows from Nevanlinna's second fundamental theorem \cite{hayman} that
$$\sum_{a\in\mathbb{C}\cup\{\infty\}}(\delta(a ,f)+\vartheta(a, f))\leq \sum_{a\in\mathbb{C}\cup\{\infty\}} \Theta(a, f)\leq 2.$$

Next we introduce the Jackson analogues of the multiplicity index and ramification index of $f$ at $a$ as in the classical Nevanlinna theory.\par

\begin{definition}
Let $f$ be a meromorphic function and $a\in\mathbb{C}\cup\{\infty\}.$ The Jackson's multiplicity index $\vartheta_{J}(a, f)$ and ramification index $\Theta_{J}(a,f)$ of $f$ at $a$ are defined respectively as $$\vartheta_{J}(a, f):=\liminf_{r\rightarrow\infty}\frac{N(r, \frac{1}{f-a})-\tilde{N}_{J}(r, f=a)}{T(r,f)}$$ and $$\Theta_{J}(a, f):=1-\limsup_{r\rightarrow\infty}\frac{\tilde{N}_{J}(r, f=a)}{T(r,f)}.$$
\end{definition}

By the second fundamental theorem for Jackson difference operator (Theorem \ref{T2}), we get the following defect relation for Jackson difference operator. The defect relations for Askey-Wilson difference operator \cite{Wilson} and Wilson difference operator \cite{Wilson2} are already given in \cite{Wilson,Wilson2}. \par

\begin{theorem}\label{T3}
Let $0<|q|<1,$ and $f$ be a nonconstant meromorphic function of zero order. Then we have $$\sum_{a\in\mathbb{C}\cup\{\infty\}}(\delta(a,f)+\vartheta_{J}(a,f))\leq\sum_{a\in\mathbb{C}\cup\{\infty\}}\Theta_{J}(a,f)\leq 2.$$
\end{theorem}

\begin{proof}
From Theorem \ref{T2}, dividing both sides of \eqref{E3.5} by the characteristic function $T(r,f)$, it yields that for any distinct values $a_{1}, a_{2},...,a_{p}\in\mathbb{C}\cup\{\infty\},$
$$(p-2)\leq \sum_{j=1}^{p}\frac{\tilde{N}_{J}(r, f=a_{j})}{T(r,f)}
+\frac{o(T(r,f))}{T(r,f)}.$$  Rearranging the terms, we then obtain $$\sum_{j=1}^{p}\Big(1-\frac{\tilde{N}_{J}(r, f=a_{j})}{T(r,f)}\Big)\leq 2+\frac{o(T(r,f))}{T(r,f)}.$$ Taking $\liminf$ on both sides as $r\rightarrow\infty,$ we have
$$\sum_{j=1}^{p}\left(\delta(a_{j}, f)+\vartheta_{J}(a_{j}, f)\right)\leq\sum_{j=1}^{p}\Theta_{J}(a_{j}, f)\leq 2.$$
\end{proof}
If $\Theta_{J}(a, f)>0$, we say that $a \in \mathbb{C}\cup\{\infty\}$ is a Jackson-Nevanlinna deficient value. From the defect relation for Jackson difference operator (Theorem \ref{T3}), we have the following result.\par

\begin{theorem} Let $f$ be a nonconstant meromorphic function with zero order. Then $f$ has at most a countable number of Jackson-Nevanlinna deficient values.
\end{theorem}

\subsection{Picard theorem for Jackson difference operator} We call $a\in \mathbb{C}\cup\{\infty\}$ a Jackson-Picard value of $f$ if
$\tilde{n}_{J}(r, f=a)=0.$ 
Since the Jackson difference operator sends a polynomial of degree $n$ to a polynomial of degree $n-1,$ we know that non-constant polynomials just have no Jackson-Picarl value unless $\infty.$ It is similar to the property for polynomials in the classical value distribution. Then we deduce the following Jackson type Picard theorem from the second fundamental theorem for Jackson difference operator (Theorem \ref{T2}). This is different from the case of the so-called Askey-Wilson-Picard value and AW-Picard theorem \cite{Wilson}.\par

\begin{theorem}\label{T4} Let $0<|q|<1,$ and let $f$ be a meromorphic function with zero order. If $f$ has three distinct Jackson-Picard values, then $f$ must be a constant.
\end{theorem}

\begin{proof}If $f$ has three distinct Jackson-Picard values $a_{1}, a_{2}, a_{3}$ (obviously, $f$ can not be a non-constant polynomial), then by the definition, we get
$\sum_{j=1}^{3}\tilde{N}_{J}(r, f=a_{j})=0.$ Assume $f$ is not a constant, then from Theorem \ref{T2}, we get
\begin{eqnarray}T(r, f)\leq o(T(r,f))\end{eqnarray}
for all $r$ on a set of logarithmic density 1, which is a contradiction.
\end{proof}

\subsection{Five-value theorem for Jackson difference operator} In 1929, R. Nevanlinna \cite{hayman} obtained the well-known five-value theorem that if two nonconstant meromorphic functions share five distinct values in $\mathbb{C}\cup\{\infty\},$ that is, the pre-images of the five points (ignoring their multiplicities) in $\mathbb{C}$ are equal, then the two functions must be identical. This has led to the development of the uniqueness problem for meromorphic functions \cite{yi-yang}. Now, we try to obtain a five-value theorem for Jackson difference operator. Before that, we need to make clear what is the meaning of two functions ``sharing" a value in the Jackson sense. \par

\begin{definition}
Let $f$ and $g$ be two nonconstant meromorphic functions, and let $a$ be a value of $\mathbb{C}\cup\{\infty\}.$ Denote by $E_{f}(a)$ the subset of $\mathbb{C}$ where $f(z)=a.$ Then we say that $f$ and $g$ share the value $a$ in the Jackson sense provided that $E_{f}(a)=E_{g}(a)$ except perhaps on the subset of $\mathbb{C}$ such that $$\tilde{N}_{J}(r, f(z)=a)-\tilde{N}_{J}(r, g(z)=a)=o\left(T(r, f)+T(r, g)\right).$$
\end{definition}

We show below a natural extension of the five-value theorem to the Jackson operator on meromorphic functions with zero order.\par

\begin{theorem}\label{T7}
Let $f$ and $g$ be two nonconstant meromorphic functions of zero order. If $f$ and $g$ share five distinct values $a_{1}$, $a_{2}$, $a_{3}$, $a_{4}$, $a_{5}$ $\in\mathbb{C}\cup\{\infty\}$ in the Jackson sense, then $f(z)\equiv g(z).$
\end{theorem}
\begin{proof}
The proof is similar to the classical one in Hayman's book\cite{hayman}. We assume the contrary that $f$ and $g$ are not identical. Applying Theorem \ref{T2} to $f,g$ and choosing $p=5$ yields $$3(T(r,f)+T(r,g))\leq\sum_{j=1}^{5}(\tilde{N}_{J}(r, f=a_{j})+\tilde{N}_{J}(r, g=a_{j}))+o(T(r,f)+T(r, g))$$ for all $r$ on a set of logarithmic density 1. Since $f$ and $g$ share the five distinct values $a_{1}$, $a_{2}$, $a_{3}$, $a_{4}$, $a_{5}$ $\in\mathbb{C}\cup\{\infty\}$ in the Jackson sense, $E_{f}(a_{j})=E_{g}(a_{j})$ except perhaps on the subset of $\mathbb{C}$ such that $$\tilde{N}_{J}(r, f(z)=a_{j})-\tilde{N}_{J}(r, g(z)=a_{j})=o\left(T(r, f)+T(r, g)\right)$$ for all $j=1, 2, \ldots, 5.$ Under the assumption that $f$ and $g$ are not identical,  we deduce that
\begin{eqnarray}\label{E5.2}
&&T(r,f)+T(r,g)\leq \frac{2}{3}\sum_{j=1}^{5}\tilde{N}_{J}(r, f=a_{j})+o\left(T(r,f)+T(r,g)\right)
\end{eqnarray} for all $r$ on a set of logarithmic density 1. At the same time, $f$ and $g$ sharing $a_1,a_2,\cdots,a_5$ in the Jackson sense implies \begin{eqnarray*}
\sum_{j=1}^{5}\tilde{N}_{J}\left(r, f=a_{j}\right)&\leq& \tilde{N}_{J}\left(r, f-g=0\right)\\&\leq& T\left(r, \frac{1}{f-g}\right)\\&\leq& T(r, f)+T(r, g)+O(1). \end{eqnarray*}
Submitting this into \eqref{E5.2} gives
\begin{eqnarray*}
\frac{1}{3}(T(r, f)+T(r, g))\leq o\left(T(r,f)+T(r,g)\right)
\end{eqnarray*} for all $r$ on a set of logarithmic density 1, which is a contradiction.
\end{proof}

\section{The Jackson kernel and two linearly independent solutions of second order Jackson difference equations}
We use $Ker(D_{q})$ to denote the kernel of Jackson difference operator $D_{q}$, where $0<|q|<1$. A meromorphic function $f$ belonging to $Ker(D_{q})$ means $D_{q}f\equiv 0$. If $f\in Ker(D_{q})$, then $f(z)\equiv f(q^kz)$ for any $k\in \mathbb{N}$. Since $q^{k}\to 0$ as $k\to+\infty,$ according to the identity theorem of holomorphic functions, we get that $f$ must be a constant. The conclusion is the same as the basic knowledge that any meromorphic function $f\in Ker(\frac{d}{dz})$ must be a constant.\vskip 2mm
\par
Let two entire functions $f_{1}$ and $f_{2}$ be linearly independent solutions of the linear second order differential equations $$f''+A(z)f=0,$$
 where $A$ is an entire function. Bank and Laine \cite{bank-laine} observed that the Wronskian determinant of $f_1,f_2$
 $$W(f_{1}, f_{2})=\left|\renewcommand\arraystretch{1.3}
 \begin{array}{cc}
                                                                                                                                          f_{1} & f_{2} \\
                                                                                                                                           f_{1}^{'} & f_{2}^{'}
                                                                                                                                         \end{array}\right|\in Ker(\frac{d}{dz}),$$                                                                                                                                       \noindent
that is $\frac{d}{dz}W(f_{1}, f_{2})\equiv 0.$ Based on the fact, they investigated the complex oscillation theory of second
order differential equations \cite{laine}.\vskip 2mm
\par
Now we define the Jackson $q$-Casorati determinant of $f_{1}$ and $f_{2}$ by
\begin{eqnarray}C_{J}(f_{1}, f_{2})=\left|\renewcommand\arraystretch{1.3}
\begin{array}{cc}
                                            f_{1} & f_{2} \\
                                            D_{q}f_{1} & D_{q}f_{2}
                                          \end{array}
\right|=f_{1}\cdot D_{q}f_{2}-f_{2}\cdot D_{q}f_{1}.
\end{eqnarray}
By \cite[Theorem 4.4.1]{bangerezako}, the linear Jackson $q$-difference equation
\begin{equation*}D_{q}^{2}f(z)+a_{1}(z)D_{q}f(z)+a_{0}(z)f(z)=0,\end{equation*}
with the coefficients $a_{1}$ and $a_{0}$ being analytic at the origin, admits two linear independent analytic solutions at the origin. Below, we show  an interesting phenomenon that for two linearly independent analytic solutions $f_1,f_2$ at the origin of the linear Jackson $q$-difference equation $$D_{q}^{2}f(z)+A(z)f(z)=0,$$
$C_J(f_1,f_2)$ does not belong to the  $Ker(D_{q})$. This is different from the case of Wronskian determinant of two linear independent
solutions for $f''-Af=0$. \par

\begin{theorem}\label{T6} Let $0<|q|<1,$  and let $A(z)\not\equiv 0$ be a non-zero functions which is analytic at the origin. If $f_{1}$ and $f_{2}$ are two linearly independent analytic solutions at the origin of the linear Jackson $q$-difference equation $$D_{q}^{2}f(z)+A(z)f(z)=0,$$
 then $C_{J}(f_{1}, f_{2})\not\in Ker D_{q}.$
\end{theorem}

To prove this theorem, we first recall that some basic properties of the Jackson difference operators (or say Jackson derivative).\par

\begin{lemma}\cite[Pages 10-11]{bangerezako}\label{L4.1} The Jackson difference operator satisfies the following rules.
\vskip 1mm
\par  (i)\,\,$D_{q}(fg)(z)=g(qz)D_{q}f(z)+f(z)D_{q}g(z)=f(qz)D_{q}g(z)+g(z)D_{q}f(z);$\vskip 1mm

\par (ii)\,\,$$D_{q}\big(\frac{f}{g}\big)(z)=\frac{g(z)D_{q}f(z)-f(z)D_{q}g(z)}{g(qz)g(z)};$$\vskip 1mm

\par (iii)\,\,\begin{equation*}\begin{split} D_{q}(f\circ g)(z)&=\frac{f(g(qz))-f(g(z))}{g(qz)-g(z)}\cdot\frac{g(qz)-g(z)}{qz-z}\\
&=:D_{q, g}f(g)\cdot D_{q, z}g(z).
\end{split}
\end{equation*}

\par (iv)\,\,\begin{equation*}D_{q, y}f^{-1}(y)=\frac{q}{D_{q, z} y},\quad \text{where}\,\,y:=f(z).\end{equation*}

\par (v)\,\,\begin{equation*} D_{q}\big[\int_{a}^{z}f(z)d_{q}z\big]=f(z),\quad
\int_{a}^{z}D_{q}f(z)d_{q}z=f(z)-f(a),
\end{equation*}
where $$\int_{a}^{z}f(t)d_{q}t:=(z-a)(1-q)\sum_{j=0}^{\infty}q^{j}f(a+q^{j}(x-a)).$$

\par(vi)\,\,\begin{equation*}
\int_{a}^{b}f(z)D_{q}g(z)d_{q}z=[fg]_{a}^{b}-\int_{a}^{b}g(qz)D_{q}f(z)d_{q}z.
\end{equation*}
\end{lemma}

\begin{lemma}\label{L4.2} Let $0<|q|<1,$ and
let $f_{1}$ and $f_{2}$ be two nonconstant functions being analytic at the origin. Then they are linearly independent if and only if $C_J(f_1,f_2)\not\equiv 0$.
\end{lemma}
\begin{proof} We first assume that $f_{1}$ and $f_{2}$ are linearly independent. If $C_J(f_{1}, f_{2})\equiv 0$, then set $F(z)=f_{1}(z)/f_{2}(z)$, we have $F(z)\equiv F(qz)$. It implies $F(z)\equiv F(q^kz)$ for any $k\in\mathbb{N}$. Since $0<|q|<1,$ by the identity theorem of meromorphic functions, we know $\frac{f_{1}(z)}{f_{2}(z)}\equiv c$, where $c$ is a constant. It is a contradiction. Hence, $C_J(f_{1}, f_{2})\not\equiv 0.$\vskip 2mm
\par On the other hand,  suppose that $C_J(f_{1}, f_{2})\not\equiv 0.$ If $f_{1}$ and $f_{2}$ are linearly dependent,
then there exists one nonzero constant $c$ such that $f_{1}(z)\equiv c f_{2}(z).$ This gives
\begin{eqnarray*}C_J(f_{1}, f_{2})(z)=\frac{1}{(q-1)z}\Big(f_{1}(z)\cdot f_{2}(qz)-f_{2}(z)\cdot f_{1}(qz)\Big)=0.
\end{eqnarray*}
for any $z$ at the neighbourhood of origin. We also obtain a contradiction. Hence $f_{1}$ and $f_{2}$ must be linearly independent.
\end{proof}

\begin{proof}[Proof of Theorem \ref{T6}] By the equation \eqref{E4.1} and Lemma \ref{L4.1}(i), we get \begin{eqnarray*}
D_{q}(C_{J}(f_{1}, f_{2})(z))&=&D_{q}(f_{1}(z)\cdot D_{q}f_{2}(z)-f_{2}(z)\cdot D_{q}f_{1}(z))\\
&=&D_{q}(f_{1}(z)\cdot D_{q}f_{2}(z))-D_{q}(f_{2}(z)\cdot D_{q}f_{1}(z))\\
&=& f_{1}(qz)\cdot D_{q}^{2}f_{2}(z)-f_{2}(qz)\cdot D_{q}^{2}f_{1}(z)\\
&=&A(z)\left(f_{2}(qz)f_{1}(z)-f_{1}(qz)f_{2}(z)\right)\\
&=&A(z)(q-1)z\cdot C_J(f_{1}(z), f_{2}(z)).
\end{eqnarray*} Since $f_{1}$ and $f_{2}$ are linearly independent, we get from Lemma \ref{L4.2} that $C_J(f_{1}, f_{2})\not\equiv 0$ and thus $D_{q}(C_{J}(f_{1}, f_{2}))\not\equiv 0.$ This means $C_{J}(f_{1}, f_{2})\not\in Ker(D_{q}).$
\end{proof}

\begin{example}\label{Ex-1} Define $[n]_{q}!=\prod_{j=1}^{n}\frac{1-q^{j}}{1-q}$, and define the $q$-version of the exponential function $exp(z)$ as
$$\exp_{q}(z)=e_{q}^{z}=\sum_{n=1}^{\infty}\frac{z^{n}}{[n]_{q}!}.$$
\cite[Corollary 2.1.1]{bangerezako} says that $e_{q}^{z}\cdot e_{q^{-1}}^{-z}=1$. Whenever $0<|q|<1, e_{q}^{z}$ is analytic in the unit disc. Define the $q$-versions of the $\sin z$ and $\cos z,$ respectively, as
\begin{eqnarray*}
\cos_{q}(z)=\frac{e_{q}^{iz}+e_{q}^{-iz}}{2}, \,\,\, \sin_{q}(z)=\frac{e_{q}^{iz}-e_{q}^{-iz}}{2i}.
\end{eqnarray*} which satisfy $D_{q}\cos_{q}(z)=-\sin_{q}(z)$ and $D_{q}\sin_{q}(z)=\cos_{q}(z)$ (\cite[Pages 23-24]{bangerezako}). One can deduce that $\sin_{q}(z)$ and $\cos_{q}(z)$ are two linearly independent solutions of the Jackson difference equation $$D^{2}_{q}f(z)+f(z)=0,$$
and  $D_{q}(C_{J}(\sin_{q}(z), \cos_{q}(z)))\not\equiv 0$ by (i) and (iii) of Lemma \ref{L4.1}.
\end{example}

\section{Entire solutions of linear Jackson difference equations}
Recall that the logarithmic order \cite{chern} of $f$ is defined by \begin{eqnarray*}
\sigma_{\log}(f)=\limsup_{r\to\infty}\frac{\log^{+} T(r, f)}{\log\log r}.
\end{eqnarray*} Any non-constant rational function is of logarithmic order one, and thus each transcendental meromorphic function has
logarithmic order no less than one. Moreover, every meromorphic function with
finite logarithmic order must have order zero. For any given $s>1$, Chern \cite[Theorem 7.4]{chern} proved that there is an entire function of logarithmic order $s$.\vskip 2mm
\par Similarly, we define the logarithmic convergent exponent of the zeros of $f$ as
\begin{eqnarray*}\lambda_{\log}(f)=\limsup_{r\to\infty}\frac{\log^{+} N(r, \frac{1}{f})}{\log\log r},
\end{eqnarray*}
and logarithmic order of the non-integral counting function $n(r, \frac{1}{f})$ is equal to $\lambda_{\log}(f)-1$ (see \cite{chern}). Chern \cite[Theorem 7.1]{chern} proved that if $f$ is a transcendental meromorphic function of finite logarithmic order, then for any two distinct $a,b\in\mathbb{C}\cup\{\infty\}$ and for any $\varepsilon>0,$ $$T(r, f)\leq N(r, \frac{1}{f-a})+N(r, \frac{1}{f-b})+O((\log r)^{\sigma_{\log}(f)-1+\varepsilon}).$$ It means that $T(r,f)$ can be estimated by $N(r,f=a)+N(r,f=b)$. Especially, if $f$ is transcendental entire and $\sigma_{\log}(f)<\infty$, then $\lambda_{\log}(f-c)=\sigma_{\log}(f)$ holds for any finite value $c$.\vskip 2mm
\par

In this section, we consider the linear difference equation of the form
\begin{equation}\label{E4.1}D_{q}^{k}f(z)+A(z)f(z)=0\end{equation}
where $k\in\mathbb{N},$ $|q|\ne 0, 1,$ and $A$ is an entire function, and obtain the following theorem.\vskip 1mm\par

\begin{theorem}\label{T5} Let $f$ be a nontrivial entire solution of the linear Jackson $q$-difference equation (\ref{E4.1}).
\par
(i).  If $A$ is a nonzero polynomial, then we get that $|q|>1,$ and that $f$ must be transcendental and satisfy $\lambda_{\log}(f)=\sigma_{\log}(f)=2.$\vskip 1mm\par

(ii).  If $A$ is a nonzero rational function $\frac{P_{1}}{P_{2}}$ where the two polynomials $P_{1}$ and $P_{2}$ are prime each other, then $f$ is either transcendental satisfying $\lambda_{\log}(f)=\sigma_{\log}(f)=2,$ or a polynomial with $\deg(P_{2})-\deg(P_{1})=k.$
\vskip 1mm\par

(iii). If $A$ is a transcendental meromorphic function with $\delta(\infty, A)>0,$ then $f$ must be transcendental and satisfy $\infty\geq\sigma_{\log}(f)\geq\sigma_{\log}(A)+1$ (we note that this inequality on growth also holds for meromorphic solutions).
\end{theorem}

From the conclusion (i) in Theorem \ref{T5}, we get the following corollary.\par

\begin{corollary}\label{C1} Let $f(\not\equiv 0)$ be a entire solution of the Jackson $q$-difference equation  \begin{eqnarray}\label{E4.4}D^{k}_{q}f(z)+A(z)f(q^{k}z)=0, \,\,\,\, (k\in\mathbb{N},\,\, 0<|q|<1),\end{eqnarray} where $A$ is a nonzero polynomial, then $f$ must be transcendental and satisfy\[\lambda_{\log}(f)=\sigma_{\log}(f)=2.\]
\end{corollary}

\begin{proof}Since $0<|q|<1,$ we get $|q^{-1}|>1.$ Note that $D_{q}f(z)=q\,D_{q^{-1}}f(qz)$ and thus $D_{q}^{k}f(z)=q^kD_{q^{-1}}^{k}f(q^{k}z).$ We can rewrite \eqref{E4.4} as
\begin{eqnarray*}D^{k}_{q^{-1}}f(q^{k}z)+q^{-k}A(q^{-k}q^{k}z)f(q^{k}z)=0, \,\,\,\, (k\in\mathbb{N},\,\, 0<|q|<1),\end{eqnarray*}
and thus,
\begin{eqnarray*}D^{k}_{q^{-1}}f(q^{k}z)+q^{-k}A(q^{-k}q^{k}z)f(q^{k}z)=0, \,\,\,\, (k\in\mathbb{N},\,\, |q^{-1}|>1).\end{eqnarray*}
Then it follows from Theorem \ref{T5}(i) that $f(q^{k}z)$ must be transcendental and satisfy $\lambda_{\log}(f(q^{k}z))=\sigma_{\log}(f(q^{k}z))=2.$ The conclusion comes from the fact that $$N(r,\frac{1}{h(qz)})=(1+o(1))N(r, \frac{1}{h}),\quad T(r, h(qz))=(1+o(1))T(r, h)$$ for any meromorphic function $h$ with zero order (refer to \cite{Barnett, chen-book, zhang-korhonen}).
\end{proof}

Noting that Wiman-Valiron theory is an important tool in the study of entire functions, we recall some definitions and basic results from Wiman-Valiron theory(see \cite{laine,chern-kim}) before proving Theorem \ref{T5}. Let $g(z)$ be a transcendental entire function with Taylor expansion $g(z)=\sum_{n=0}^{\infty}a_{n}z^{n}.$ The maximum term $\mu(r, g)$ and the central index $\nu(r, g)$ of $g$ are defined, respectively, by
$$\mu(r, g)=\max_{n\geq 0}\{|a_{n}|r^{n}\}\quad \text{and}\quad \nu(r, g)=\max\{m: |a_{m}|r^{m}=\mu(r, g)\}.$$
The order and logarithmic order of $g$ can be defined equivalently by $$\sigma(g)=\limsup_{r\to\infty}\frac{\log^{+}\nu(r, g)}{\log r},\quad \sigma_{\log}(g)=\limsup_{r\to\infty}\frac{\log^{+}\nu(r, g)}{\log\log r}+1.$$
\par
By Wen and Ye's Wiman-Valiron theorem for $q$-difference \cite{wen-ye}, we obtain a Wiman-Valiron theorem for Jackson difference.
\begin{lemma}\label{L4.4}
Suppose that $k$ is a positive integer, $q$ is a complex number with $q^{k}\in\mathbb{C}\setminus\{0, 1\}$. Let $f$ be a transcendental entire function of order strictly less than $1/2$ and $F\subset \mathbb{R}^{+}$ a set of finite logarithmic measure. Then for any $0<\delta<1/4$ and any $z$ with $|z|=r\not\in F$ satisfying $$|f(z)|>M(r, f)\nu(r, f)^{\delta-\frac{1}{4}},$$ we have
\begin{eqnarray*}
\frac{D^{k}_{q}f(z)}{f(z)}=
(q-1)^{-k}z^{-k}q^{-\frac{k(k-1)}{2}}\sum_{j=0}^{k}(-1)^{j}\left[\begin{array}{c}
                                                                                            k \\
                                                                                            j
                                                                                          \end{array}
\right]_{q}q^{\frac{j(j-1)}{2}}e^{(q^{k-j}-1)\nu(r, f)(1+o(1))}.\end{eqnarray*}Particularly, $k=1,$  \begin{eqnarray*}
\frac{D_{q}f(z)}{f(z)}=\frac{1}{(q-1)z}\left(e^{(q-1)\nu(r, f)(1+o(1))}-1\right).
\end{eqnarray*}
\end{lemma}

\begin{proof}\cite[Theorem 2.3]{wen-ye} says that for any $0<\delta<\frac{1}{4}$ and any $z$ with $|z|=r\not\in F$ satisfying $|f(z)|>M(r, f)\nu(r, f)^{\delta-\frac{1}{4}}$, we have $$\frac{f(q^{k}z)}{f(z)}=e^{(q^{k}-1)\nu(r, f)(1+o(1))}.$$
Hence,  for $k=1,$  we get immediately that \begin{equation}\label{wV-q-difference}
\frac{D_{q}f(z)}{f(z)}=\frac{1}{(q-1)z}\left(\frac{f(qz)}{f(z)}-1\right)=\frac{1}{(q-1)z}\left(e^{(q-1)\nu(r, f)(1+o(1))}-1\right).
\end{equation}
For general $k\geq 1,$  recall the equlity\cite{hahn} (see also \cite[page 13]{bangerezako} and \cite[Lemma 2.2]{annaby-mansour}) \begin{equation}\label{exact-form-k-Jackson}
D^{k}_{q}f(z)=(q-1)^{-k}z^{-k}q^{-\frac{k(k-1)}{2}}\sum_{j=0}^{k}(-1)^{j}\left[\begin{array}{c}
                                                                                            k \\
                                                                                            j
                                                                                          \end{array}
\right]_{q}q^{\frac{j(j-1)}{2}}f(q^{k-j}z),
\end{equation}
then combining this with \eqref{wV-q-difference} yields
\begin{equation*}
\begin{split}
\frac{D^{k}_{q}f(z)}{f(z)}=&(q-1)^{-k}z^{-k}q^{-\frac{k(k-1)}{2}}\sum_{j=0}^{k}(-1)^{j}\left[\begin{array}{c}
                                                                                            k \\
                                                                                            j
                                                                                          \end{array}
\right]_{q}q^{\frac{j(j-1)}{2}}\frac{f(q^{k-j}z)}{f(z)}\\
=&(q-1)^{-k}z^{-k}q^{-\frac{k(k-1)}{2}}\sum_{j=0}^{k}(-1)^{j}\left[\begin{array}{c}
                                                                                            k \\
                                                                                            j
                                                                                          \end{array}
\right]_{q}q^{\frac{j(j-1)}{2}}e^{(q^{k-j}-1)\nu(r, f)(1+o(1))}.
\end{split}
\end{equation*}
\end{proof}

We next prove a lemma of logarithmic difference for meromorphic functions, in which the first equality with $k=1$ can be also seen in (\cite[Theorem 2.2]{wen}).\par

\begin{lemma}\label{L4.5} Let $f$ be a nonconstant meromorphic function with finite logarithmic order, $k\in\mathbb{N},$ and $q\in\mathbb{C}\setminus\{0\}.$ Then for any $\varepsilon$, we have
\begin{equation}
\begin{split}
m\left(r, \frac{f(q^{k}z)}{f(z)}\right)&=O\left((\log r)^{\sigma_{\log}(f)-1+\varepsilon}\right),\\
m\left(r, \frac{D_{q}^{k}f(z)}{f(z)}\right)&=O\left((\log r)^{\sigma_{\log}(f)-1+\varepsilon}\right).
\end{split}
\end{equation}
\end{lemma}
\begin{proof} Without loss of generality, we assume that $f(0)\neq 0, \infty.$ Or else, we always find a suitable $m\in\mathbb{Z}$ such that $g(z)=z^m f(z)$ satisfying $f(0)\not=0,\infty$. It follows from \cite[Lemma 5.1]{Barnett} that
\begin{eqnarray*}
&&
m\left(r, \frac{f(qz)}{f(z)}\right)\\&\leq&\left(n(\rho, f)+n(\rho,  \frac{1}{f})\right)\left(\frac{|q-1|^{\delta}(|q|^{\delta}+1)}{\delta(\delta-1)|q|^{\delta}}+\frac{|q-1|r}{\rho-|q|r}+\frac{|q-1|r}{\rho-r}\right)\\
&&+\frac{4|q-1|r\rho}{(\rho-r)(\rho-|q|r)}\left(T(\rho, f)+\log^{+}\left|\frac{1}{f(0)}\right|\right),
\end{eqnarray*} where $|z|=r,$ $\rho\leq\max\{r, |q|r\}$ and $0<\delta<1.$ We observe that the counting function of poles satisfies
\begin{equation*}
\begin{split}
N(\rho^{2}, f)&\geq\int_{\rho}^{\rho^{2}}\frac{n(t, f)-n(0, f)}{t}dt+n(0, f)\log \rho^{2}\geq n(\rho, f)\log \rho.
\end{split}
\end{equation*}We adopt the idea of Chiang-Feng \cite{Wilson2, Feng} to take $\rho=r\log r,$ and then get that
\begin{equation}\label{E4.3}
\begin{split}&
m\left(r, \frac{f(qz)}{f(z)}\right)\\\leq  &O(1)\left(N(\rho^{2}, f)+N(\rho^{2}, \frac{1}{f})\right)\left(1+\frac{1}{\log r-|q|}+\frac{1}{\log r-1}\right)\frac{1}{\log r}\\
&+\frac{4|q-1|\log r}{(\log r-1)(\log r-|q|)}\left(T(\rho, f)+O(1)\right).
\end{split}
\end{equation}
Since $\sigma_{\log}(f)<\infty$, then for any $\varepsilon>0,$ we have
\begin{equation*}
\begin{split}
\max\{T(\rho, f),\,N(\rho^2,f),\,N(\rho^2,1/f)\}\leq &(\log (r\log r)^2)^{\sigma_{\log}(f)+\frac{\varepsilon}{2}}\\=&O\left((\log r)^{\sigma_{\log}(f)+\varepsilon}\right).\end{split}
\end{equation*}Submitting these inequalities into \eqref{E4.3} yields
\begin{eqnarray*}
m\left(r, \frac{f(qz)}{f(z)}\right)=O\left((\log r)^{^{\sigma_{\log}(f)-1+\varepsilon}}\right),
\end{eqnarray*}
which is also obtained by Wen \cite[Theorem 2.2]{wen} for $\rho=r^2$. For each meromorphic function $f$, we have $T(\frac{r}{q},f(qz))=T(r,f)$. This implies $\sigma_{\log}(f(qz))=\sigma_{\log}(f)<\infty.$ It is not difficult to see that \begin{equation*}
m\left(r,\frac{f(q^{k}z)}{f(z)}\right)\leq \sum_{j=0}^{k-1} m\left(r, \frac{f(q^{j+1}z)}{f(q^{j}z)}\right)=O\left((\log r)^{^{\sigma_{\log}(f)-1+\varepsilon}}\right).
\end{equation*}
From this and \eqref{exact-form-k-Jackson}, it follows that
\begin{eqnarray*}m\left(r, \frac{D^{(k)}_{q}f(z)}{f(z)}\right)\leq \sum_{j=0}^{k}m\left(r, \frac{f({q}^{k-j}z)}{f(z)}\right)+O(1)=O\left((\log r)^{^{\sigma_{\log}(f)-1+\varepsilon}}\right).
\end{eqnarray*}
\end{proof}

\begin{proof}[Proof of Theorem \ref{T5}]
(i). Since $A$ is a nonzero polynomial of degree $n$, we write it as $A(z)=a_nz^n+\cdots$ where $n$ is its degree. Bergweiler, Ishizaki and Yanagihara \cite{b-i-y} proved that all meromorphic solutions of the general $q$-difference equations \begin{eqnarray}\label{E5.4}\sum_{j=0}^{m}b_{j}(z)f(c^{j}z)=Q(z) \quad (0<|c|<1)\end{eqnarray}  satisfy $T(r, f)=O((\log r)^{2}),$ where the coefficients $b_{0}(\not\equiv 0), \ldots, b_{m}(\equiv 1)$ and $Q$ are rational functions. Thus by this result, we know that any nonzero entire solution $f$ of \eqref{E4.1} satisfies $\sigma(f)=0,\,\sigma_{\log}(f)\leq 2$. Since
the $k$-th Jackson difference operator sends a polynomial of degree $m$ to a polynomial of degree $\max\{m-k,0\}$, then every entire solution $f$ of \eqref{E4.1} must be transcendental. Rewrite \eqref{E4.1} as $-A(z)=\frac{D_{q}^{k}f(z)}{f(z)}$. By Lemma \ref{L4.4}, for any $z$ with $|z|=r\not\in F$ satisfying $|f(z)|=M(r, f),$ we have
\begin{equation}\label{N}
\begin{split}
(q-1)^{-k}z^{-k}q^{-\frac{k(k-1)}{2}}&\sum_{j=0}^{k}(-1)^{j}
\left[\begin{array}{c}
                                                                                            k \\
                                                                                            j
                                                                                          \end{array}
\right]_{q}q^{\frac{j(j-1)}{2}}e^{(q^{k-j}-1)\nu(r, f)(1+o(1))}\\
&=-a_{n}z^{n}(1+o(1)).
\end{split}
\end{equation}
where $F$ is a set of finite logarithmic measure. This implies that $|q|>1$ since $\nu(r,f)\to\infty$ as $r\to\infty$. Further, set  $b=(q-1)^{-k}q^{-\frac{k(k+1)}{2}}\left[\begin{array}{c}
k\\
0
\end{array}\right]_q$, then
\begin{eqnarray*}b(1+o(1))e^{(q^k-1)\nu(r, f)})=|a_n|r^{n+k}(1+o(1)),\end{eqnarray*}
so $\log \nu(r, f)=\log\log r+O(1)$ holds for any $r\not\in F$. Hence
$$\sigma_{\log}(f)-1=\limsup_{r\to\infty}\frac{\log^{+} \nu(r, f)}{\log \log r}=1.$$  Since the logarithmic order of $f$ is finite, we have $\lambda_{\log}(f)=\sigma_{\log}(f)=2.$\vskip 2mm
\par (ii). For nonzero rational $A=\frac{P_1}{P_2},$ we write it as $A(z)=c\,z^{\deg(P_1)-\deg(P_2)}$ for large $z$. Also from the result of Bergweiler, Ishizaki and Yanagihara \cite{b-i-y}, any nonzero entire solution $f$ of \eqref{E4.1} has $\sigma(f)=0,\sigma_{\log}(f)\leq 2$. If $f$ is a nonzero polynomial, then by the basic property of $D_q^k$ on polynomials, we know $\deg(P_{2})-\deg(P_{1})=k.$ Now we treat the case that $f$ is a transcendental entire solution. Again by Lemma \ref{L4.4}, similarly for any $z$ with $|z|=r\not\in F$ satisfying $|f(z)|=M(r, f)$, we have
\begin{equation*}
\begin{split}
-(q-1)^{-k}z^{-k}q^{-\frac{k(k-1)}{2}}&\sum_{j=0}^{k}(-1)^{j}
\left[\begin{array}{c}
                                                                                            k \\
                                                                                            j
                                                                                          \end{array}
\right]_{q}q^{\frac{j(j-1)}{2}}e^{(q^{k-j}-1)\nu(r, f)(1+o(1))}
=\frac{P_1(z)}{P_2(z)}.
\end{split}
\end{equation*}
From this equality, it is follows that for $0<|q|<1$, we have
\begin{eqnarray*}b^{-1}(1+o(1))e^{(1-q^k)\nu(r,f)}=c^{-1}z^{\deg(P_{2})-k-\deg(P_{1})},\end{eqnarray*}
and for $|q|>1$, we have
\begin{eqnarray*}b(1+o(1))(e^{(q^k-1)\nu(r, f)})=c\,z^{\deg(P_{1})+k-\deg(P_{2})}.\end{eqnarray*}
This means that in the above two cases on $|q|$, $\log \nu(r, f)=\log\log r+O(1),\quad r\not\in F$.
Then by the equivalent definition of $\sigma_{\log}(f)$, $\sigma_{\log}(f)=2$ holds.\vskip 2mm
\par

(iii). Assume that $A$ is a transcendental meromorphic, and $\delta(\infty, A)>0$. It means that for enough large $r$,
$$N(r, A)<(1-\frac{\delta(\infty, A)}{2})T(r, A),\,\, \text{that is}\,\, \frac{\delta(\infty,A)}{2}T(r,A)\leq m(r,A).$$
In this case, the solution $f$ of \eqref{E4.1} must be transcendental. If $f$ is of finite logarithmic order, then by Lemma \ref{L4.5}), for any $\varepsilon>0$, we have
\begin{eqnarray*}
\frac{\delta(\infty,A)}{2}T(r, A)\leq m(r,A)\leq m\Big(r, \frac{D_{q}^{k}f(z)}{f(z)}\Big)=O\left((\log r)^{\sigma_{\log}(f)-1+\varepsilon}\right).
\end{eqnarray*}
This implies $\sigma_{\log}(f)\geq \sigma_{\log}(A)+1$. It is trivial that $\sigma_{\log}(f)\leq\infty.$ Therefore, $\sigma_{\log}(A)+1\leq\sigma_{\log}(f)\leq\infty$. In addition, it is clear that the derivation of this inequality on growth also holds for meromorphic solutions.
\end{proof}
\par Modifying the proof of Theorem \ref{T5}(i) a little, we can also discuss non-homogeneous linear Jackson $q$-difference equations with polynomial coefficients.
\begin{theorem}
Let $f$ be a transcendental entire solution of the linear $q$-difference equation $D_q^kf(z)+A(f)f(z)=B(z)$, where $A,B$ are non-zero polynomials. Then $|q|>1$, and $\lambda_{\log}(f)=\sigma_{\log}(f)=2.$
\end{theorem}
\begin{proof}
Similarly as in the proof of Theorem \ref{T5}(i), from Bergweiler, Ishizaki and Yanagihara \cite{b-i-y} on \eqref{E5.4}, we have
 $\sigma_{\log}(f)\leq 2$, and further $\lambda_{\log}(f)=\sigma_{\log}(f)$.\vskip 2mm
 \par Since $f$ is transcendental, clearly $|\frac{M(r,B)}{M(r,f)}|\to 0$ as $r\to\infty$. Applying Lemma \ref{L4.4} to $-A(z)=\frac{D_{q}^{k}f(z)}{f(z)}+\frac{B(z)}{f(z)}$, it yields that for any $z$ satisfying $|f(z)|=M(r, f)$ with $|z|=r$ outside a set $F$ of finite logarithmic measure, we can obtain an equality similar to \eqref{N}. This implies that $|q|>1$ and $\lambda_{\log}(f)=\sigma_{\log}(f)=2.$\end{proof}
At the end, we give some examples to investigate the growth of some $q$-special functions by Theorem \ref{T5}.\par

\begin{example}\label{Ex2} Set \cite[Page 19]{bangerezako} $$\tilde{e}_{q}(z):=\tilde{e}_{q}^{z}=_1\phi_{0}(0; -; q, z)=\sum_{n=0}^{+\infty}\frac{z^{n}}{(q; q)_{n}},$$  then $e_{q}^{z}=\tilde{e}_{q}((1-q)z)$ and $\tilde{e}_{q}(z)\tilde{e}_{q^{-1}}(q^{-1}z)=1,$ see \cite[Proposition 5.2]{ramis}. For $|q|>1,$ $$\tilde{e}_{q}(z)=\prod_{n=1}^{+\infty} (1-q^{-n}z)=(q^{-1}z; q^{-1})_{\infty}$$
 is an entire function, and from \cite[Proposition 5.2]{ramis}, $\tilde{e}_{q}(z)$ satisfies the linear first order Jackson $q$-difference equation $$D_{q}f(z)+\frac{1}{q-1}f(z)=0, \quad (|q|>1).$$ By Theorem \ref{T5} (i), we get that $\tilde{e}_{q}(z)$ $(|q|>1)$ is of logarithmic order two.
\end{example}

\begin{example} \label{Ex3} Suppose that the function $\exp_{q}(\lambda z)$ is a solution of the equation
\begin{equation}\label{E5.3}D^{2}_{q}f(z)+f(z)=0.\end{equation} Then see \cite[Page 45]{bangerezako}, $\lambda^{2}+1=0$ is said to be the characteristic equation of \eqref{E5.3}, and $f_{1}(z)=\exp_{q}(iz)$ and $f_{1}(z)=\exp_{q}(-iz)$ are two independent solutions of
\eqref{E5.3}. Let $|q|>1.$ Then from Theorem \ref{T5}(i), we get that the two solutions are of logarithmic order two.
\end{example}

\begin{example} \label{Ex3'} From Example \ref{Ex-1} or Example \ref{Ex3}, we get that whenever $|q|>1,$ the function $\sin_{q}z$ and $\cos_{q}z$ are also two independent solutions of the equation
\begin{equation} D^{2}_{q}f(z)+f(z)=0.\end{equation} Then from Theorem \ref{T5}(i), $\sin_q(z)$ and $\cos_q(z)$ are of logarithmic order two.
\end{example}

\begin{example}\label{Ex4} Denote by \cite[Page 19]{bangerezako} $$E_{q}^{z}=_0\phi_{0}(-; -; q,-z)=\sum_{n=0}^{+\infty}\frac{q^{\frac{n(n-1)}{2}}z^{n}}{(q; q)_{n}}.$$ We have $\tilde{e}_{q}^{z}E_{q}^{-z}=1$ and $\tilde{e}_{q^{-1}}(q^{-1}z)=E_{q}^{-z},$ see \cite[Page 74]{ramis} or \cite[Corollary 2.1.2]{bangerezako}. From \cite[Proposition 5.1]{ramis}, we know that if $|q|<1,$ the entire function $E_{q}^{z}=\prod _{n=0}^{+\infty}(1+q^{n}z)=(-z; q)_{\infty}$ satisfies the first order Jackson $q$-difference equation $$D_{q}f(z)+\frac{1}{(q-1)(z+1)}f(z)=0, \quad (0<|q|<1).$$ By Theorem \ref{T5}(ii), we get that $E_{q}^{z}$ $(0<|q|<1)$ is of logarithmic order two.
\end{example}

\begin{example} \label{Ex5}Let $|q|\neq 0, 1.$ We observe that the polynomial $P(z)=z^{5}+1$ satisfies a first order Jackson $q$-difference equation
$$D_{q}f(z)-\frac{(q^{5}-1)z^{4}}{(q-1)(z^{5}+1)}f(z)=0$$
and the second order Jackson $q$-difference equation
$$D_{q}^{2}f(z)-\frac{(q^{9}-q^{5}-q^{4}+1)z^{3}}{(q-1)^{2}(z^{5}+1)}f(z)=0.$$
We rewrite $$-\frac{(q^{5}-1)z^{4}}{(q-1)(z^{5}+1)}=\frac{P_{1}}{P_{2}},\quad -\frac{(q^{9}-q^{5}-q^{4}+1)z^{3}}{(q-1)^{2}(z^{5}+1)}=\frac{Q_{1}}{Q_{2}},$$
clearly $\deg(P_{1})=4=\deg(P_{2})-1$ and $\deg(Q_{1})=2=\deg(Q_{2})-2.$ This shows that for a polynomial solution, the conclusion in Theorem \ref{T5} (ii) is sharp.
\end{example}

\begin{example}\label{Ex6} From \cite[Page 16]{bangerezako}, we see that every entire solution of Jackson $q$-difference equation
\begin{equation}\label{E5.5} D_{q}f(z)=P(z)f(qz),\,\,\,(0<|q|<1)\end{equation} with polynomial coefficient $P(z)$ reads $$f(z)=f(0)\prod_{j=0}^{\infty}\left(1+(1-q)q^{j}zP(q^{j}z)\right).$$
From Corollary \ref{C1}, it follows that $\sigma_{\log}(f)=2.$ Especially, when $P(z)$ is a nonzero constant $a,$ then it is known \cite[Page 18]{bangerezako} that the \eqref{E5.5} has a solution of the form $$f(z)=f(0)\exp_{q^{-1}}(az)=f(0)\sum_{n=0}^{\infty}\frac{a^{n}z^{n}}{[n]_{q^{-1}}!},$$ where $[n]_{q^{-1}}!$ is obtained from $[n]_{q}!$ by replacing $q$ by $q^{-1}.$
\end{example}

Finally, we propose two interesting problems deserved to be further studied.\par

\begin{question} Set $g(z)=\frac{1}{f(z)}$, where $f(z)$ appears in the above examples, then $g(z)$ is meromorphic in the plane and may satisfy some Jackson $q$-difference equation with rational coefficients. How about the growth and the distribution of zeros and poles of meromorphic solutions of Jackson $q$-difference equations such as
$$D^{k}_{q}f(z)+A(z)f(z)=B(z),$$
where $A(z),B(z)$ are rational.
\end{question}
\begin{question}In Theorem \ref{T5} (iii), when $A$ in \eqref{E4.1} is transcendental, the low estimate of the growth of meromorphic solutions is studied (implied in its proof). What is the upper estimate of the growth of these solutions? Is there an entire or meromorphic solution with infinite logarithmic order for \eqref{E4.1}? Further, for Jackson $q$-difference equations such as
$$D^{k}_{q}f(z)+A_{k-1}(z)D^k_qf(z)+\cdots+A_0(z)f(z)=B(z),$$
where $A_{k-1},\cdots,A_1,A_0$ and $B$ are meromorphic in the plane, what can we say about the meromorphic solutions?
\end{question}

\end{document}